\documentclass[12pt, reqno]{amsart}
\usepackage{amsmath, amsthm, amscd, amsfonts, amssymb, graphicx, color}
\usepackage{setspace}
\usepackage{multicol}
\usepackage[bookmarksnumbered, colorlinks, plainpages]{hyperref}
\hypersetup{colorlinks=true,linkcolor=red, anchorcolor=green, citecolor=cyan, urlcolor=red, filecolor=magenta, pdftoolbar=true}

\newtheorem{theorem}{Theorem}[section]
\newtheorem{lemma}[theorem]{Lemma}

\theoremstyle{definition}

\theoremstyle{remark}

\numberwithin{equation}{section}

\newcommand{\NN}{\mathbb{N}}

\newcommand{\CC}{\mathbb {C}}

\begin{document}
\setcounter{page}{1}
\title[Spectral properties  of Volterra-type integral operators on Fock--Sobolev  spaces]{Spectral properties of  Volterra-type integral  operators on Fock--Sobolev  spaces }
\author [Tesfa  Mengestie]{Tesfa  Mengestie }
\address{Department of Mathematical Sciences \\
Western Norway University of Applied Sciences\\
Klingenbergvegen 8, N-5414 Stord, Norway}
\email{Tesfa.Mengestie@hvl.no}

\thanks{The  author  is partially supported by HSH grant 1244/ H15.}
\begin{abstract}
We study some spectral properties of Volterra-type integral operators $V_g$ and $I_g$  with holomorphic symbol $g$ on the  Fock--Sobolev spaces $\mathcal{F}_{\psi_m}^p$. We showed that $V_g$ is bounded on $\mathcal{F}_{\psi_m}^p$ if and only if $g$ is a complex polynomial of degree not
exceeding two, while compactness of $V_g$ is described by degree of $g$ being not bigger than one. We also identified  all those  positive numbers $p$ for which the operator $V_g$ belongs to the Schatten  $\mathcal{S}_p$ classes. Finally, we characterize the spectrum of  $V_g$ in terms of a closed disk of radius twice   the coefficient of the highest degree  term in a polynomial expansion of  $g$.
\end{abstract}

\maketitle

\section{Introduction} \label{1}
The boundedness and compactness properties of  integral  operators are very well studied objects in operator related function-theories.  They have been studied for a broad class of operators on various spaces of holomorphic  functions including the Hardy spaces \cite{ALC, Alsi1,Jord}, Bergman spaces \cite{JPP,JPP1,JR}, Fock spaces \cite{Olivia1, Olivia, TM1, TM, TM0,TM2, TM3}, Dirichlet spaces \cite{Alsi2,GGP,GP}, Model spaces \cite{TMM},  and logarithmic Bloch spaces \cite{XH}. Yet, they still constitute an active area of research because of their multifaceted implications. Typical examples of operators  subjected  to  this phenomena are  the Volterra-type integral operator $V_g$ and its companion $I_g$,  defined by
 \begin{align*}
 V_gf(z)= \int_0^z f(w)g'(w) dw \ \ \ \text{and} \ \  \ I_gf(z)= \int_0^z f'(w)g(w) dw,
\end{align*} where $g$ is a holomorphic symbol. Applying integration by parts in any one of the above integrals
gives the relation
 \begin{align}
\label{parts}V_g f+ I_g f= M_g f-f(0)g(0),
\end{align} where $M_g
f= g f$ is the multiplication operator of symbol $g$. On   the classical Fock spaces with the  Gaussian weight, some spectral structures of   these operators were studied by several  authors for example  in \cite{Olivia1,TM,TM1,TM0,TM2}. On the other hand, when the weight decays faster than the classical Gaussian  weight, they were recently studied in \cite{Olivia, TM3}. From the results in these two later works, we observed that  while the operator $V_g$ enjoys a richer structure when it acts between weighted Fock spaces of faster decaying weights  in contrast to its action on  the classical Fock spaces, the analogues  structures for $I_g$ and $M_g$ has  got rather poorer.  A natural question is then what happens  to these structures when the weight decays slower than the classical Gaussian  weight?  The central aim of this paper is to investigate this situation. Prototype  examples of  spaces generated by such  slower decaying weights, which we are interested in,  are  the Fock--Sobolev spaces as described below.

 Let $m$  be  any nonnegative  integer and $0<p<\infty$. Then,  the Fock--Sobolev spaces  $\mathcal{F}_{(m,p)}$ consist of entire functions $f$ such that $f^{(m)}$, the m-th order derivative of $f$, belongs to the classical Fock spaces $\mathcal{F}_p;$
 which consist of all entire functions $f$ for which
\begin{align*}
\int_{\CC}
|f(z)|^p  e^{-\frac{p}{2}|z|^2} dA(z) <\infty.
\end{align*}
 The Fock--Sobolev spaces were introduced in \cite{RCKZ} and  it was proved that  $f$ belongs to $\mathcal{F}_{(m, p)}$ if and only if  the function $z \mapsto |z|^m f(z)$ belongs to $L^p(\CC, e^{-p|z|^2 /2})$. By  closed graph theorem argument, we have that $f$ belongs to $\mathcal{F}_{(m, p)}$ if and only if $z \mapsto (\beta+ |z|)^m f(z)$ belongs to  $L^p(\CC, e^{-p|z|^2 /2})$ for any positive number $\beta.$   A consequence of this is that the norm in  $\mathcal{F}_{(m,p)}$ is comparable to the quantity
\begin{equation*}
 \bigg(C_{(p,m)} \int_{\CC}
|f(z)|^p (1+|z|)^{mp} e^{-\frac{p}{2} |z|^2 } dA(z)\bigg)^{1/p}
\end{equation*}
for $ 0 < p <\infty$,  and
\begin{equation*}
C_{(m, p)}= ( p/2)^{\frac{mp}{2}+ 1}\big(
\pi \Gamma Γ\Big(\frac{mp}{2}+ 1\big)\Big)^{-1},
\end{equation*} where $ \Gamma$ denotes the Gamma function,  $dA$ denotes the
usual Lebesgue area  measure on $\CC$, and  we  fix  $\beta \simeq  1$  for simplicity.
  To put the spaces into weighted/generalized  Fock spaces context,  we may now set the sequence of the corresponding weight functions as
\begin{align}
\label{weight}
\psi_m(z)= \frac{1}{2}|z|^2-m\log(1+ |z|),
 \end{align}
and observe that the
 Fock--Sobolev spaces $\mathcal{F}_{(m, p)}$ are just the weighted Fock spaces   $\mathcal{F}_{\psi_m}^p$ which consist of all entire functions $f$ for which \footnote{The  notation $U(z)\lesssim V(z)$ (or
equivalently $V(z)\gtrsim U(z)$) means that there is a constant
$C$ such that $U(z)\leq CV(z)$ holds for all $z$ in the set of a
question. We write $U(z)\simeq V(z)$ if both $U(z)\lesssim V(z)$
and $V(z)\lesssim U(z)$.}
\begin{align*}
\int_{\CC}
|f(z)|^p  e^{-p\psi_m(z) } dA(z)\simeq  \|f\|_{(p,m)}^p <\infty. \end{align*}
We may  now state our first  main result.
\begin{theorem}\label{thm1}
Let $g$ be an entire function on $\CC$ and $0<p,q< \infty$. Then if
\begin{enumerate}
\item $0<p\leq q< \infty$, then $V_g: \mathcal{F}_{\psi_m}^p \to \mathcal{F}_{\psi_m}^q$ is
\begin{enumerate}
\item  bounded if and only if $g (z)= az^2+bz+c,  \ a, b, c\in \CC$.
\item  compact if and only if $g (z)= az+b, \  a, b, \in \CC $.
\end{enumerate}
\item $0<q<p< \infty$, then  the following statements are equivalent
\begin{enumerate}
\item $V_g: \mathcal{F}_{\psi_m}^p \to \mathcal{F}_{\psi_m}^q$ is   bounded;
\item $V_g: \mathcal{F}_{\psi_m}^p \to \mathcal{F}_{\psi_m}^q$ is   compact;
\item $g(z)= az+b$  whenever  $\frac{q}{2}>\frac{p-q}{p}$, and g= constant otherwise.
\end{enumerate}
\item $0<p<\infty$ and $V_g$ compact on $ \mathcal{F}_{\psi_m}^2$, then $V_g$  belongs to the Schatten $\mathcal{S}_p(\mathcal{F}_{\psi_m}^2)$ classes for all $p>2$. On the other hand, if $0<p<2,$  then $V_g$ belongs to $\mathcal{S}_p(\mathcal{F}_{\psi_m}^2)$ if and only if $g$ is the zero function.
\end{enumerate}
\end{theorem}
\begin{theorem}\label{thm2}
Let $g$ be an entire function on $\CC$ and $0<p,q< \infty$. Then if
\begin{enumerate}
\item  $0<p\leq q< \infty$, then $I_g: \mathcal{F}_{\psi_m}^p \to \mathcal{F}_{\psi_m}^q$ is
\begin{enumerate}
\item  bounded if and only if $g$ is a constant function.
\item  compact if and only if  $g$ is the  zero function.
\end{enumerate}
\item  $0<q<p< \infty$, then the following are equivalent.
\begin{enumerate}
\item $I_g: \mathcal{F}_{\psi_m}^p \to \mathcal{F}_{\psi_m}^q$ is bounded;
\item $I_g: \mathcal{F}_{\psi_m}^p \to \mathcal{F}_{\psi_m}^q$ is compact;
\item  $g$ is the zero function.
\end{enumerate}
\end{enumerate}
\end{theorem}
It may be noted that
when $m=0,$ the spaces $\mathcal{F}_{\psi_m}^p $ reduce to the classical Fock  spaces $\mathcal{F}^p,$  and for this particular case, the results were proved in  \cite{Olivia1,TM1,TM}. In view of our current results, we conclude that there exists no richer boundedness and  compactness structures for $V_g$ and $I_g$ on Fock-Sobolev spaces  than those on the classical setting.  As can be seen from \eqref{weight}, the Fock--Sobolev spaces are generated by making small perturbations of the weight function on the classical Fock spaces. It turns out that such perturbations
play no role in the structure of the operators   and rather extend the classical  results to all  the spaces $\mathcal{F}_{\psi_m}^p $  independent of the values of  $m$.

In addition, the results show that there exists no nontrivial Volterra companion integral type operators $I_g$ acting between any of the Fock--Sobolev spaces.

 Another observation worthwhile making is that  when $g(z)= z,$ the operator $V_g$ reduces to the original Volterra operator $Vf(z)= \int_{0}^z f(w)dA(w).$ By particular cases of  the results above, we conclude that this operator is always bounded in its action on the Fock--Sobolev spaces.
  \subsection{Spectrum of the integral operators}
In contrast to the fairly good understanding of the boundedness, compactness, and Schatten class membership of the Volterra-type integral operators on various Banach spaces, much less is known about their spectral. Recently, Constantin and  Persson \cite{Olivia2}, determined the spectrum of $V_g$  acting on generalized  Fock spaces  where the  inducing weight function  takes the particular  form $|z|^A, \ A>0$ and $1\leq p<\infty$.  Our next result describes the spectrum of the Volterra-type integral operators on Fock--Sobolev spaces  in terms of a closed disk of radius involving the coefficient of the highest degree term in a polynomial expansion of  $g$ as precisely formulated below.
\begin{theorem} \label{thm3} (i) Let $p\geq 1$ and $V_g:\mathcal{F}_{\psi_m}^p \to \mathcal{F}_{\psi_m}^p$ be  a bounded operator, i.e.  $g(z)= az^2+bz,\ \   a, b, \in \CC$. Then
\begin{align}
\label{spectrum}
\sigma(V_g)= \big\{\lambda \in \CC: |\lambda|\leq 2|a|\big \}= \{0\}\cup \overline{\big \{\lambda \in \CC\setminus\{0\}:e^{g(z)/\lambda} \notin \mathcal{F}_{\psi_m}^p \big \} }.
\end{align}
\item (ii) Let $p\geq 1$ and $I_g:\mathcal{F}_{\psi_m}^p \to \mathcal{F}_{\psi_m}^p$ be  a bounded operator, i.e., $g(z)= c=$constant. Then
 \begin{align*}
\sigma(I_g)= \{c\}.
\end{align*}
\end{theorem}
The results her are also independent of the order $m$, and  coincide  with  the corresponding results in the classical Fock spaces setting. Furthermore, for $ m= 0$,  the result in \eqref{spectrum} follows also from the main result in \cite{Olivia2} as a particular case.
\section{Preliminaries}
For each $m$, the spaces $\mathcal{F}_{\psi_m}^2 $  are  reproducing kernel Hilbert
spaces with  kernel $K_{(w,m)}$ and normalized reproducing kernel functions $k_{_{(w,m)}}$  for a  point $w$ in $\CC.$  An explicit expression for
$K_{(w,m)}$ is still unknown. On the other hand, for each $w$ in $\CC$ by Proposition~2.7 of  \cite{CCK},  we have an important asymptotic relation
\begin{equation}
\label{asymptotic}
\|K_{_{(w,m)}}\|_{(2, m)}^2 \simeq  e^{2\psi_m(w)}.
\end{equation}
 As noted before when $m=0,$ the space $\mathcal{F}_m^2$ reduces to the classical Fock  space $\mathcal{F}^2$, and in this  case  we  precisely have
 $\|K_{_{(z,0)}}\|_{(2, 0)}^2= e^{|z|^2}$ and $K_{(w,0)}(z)= e^{\overline{w}z}$. For other p's,   Corollary~14 of \cite{RCKZ} gives the one sided estimate
\begin{align}
\label{forall}
\|K_{(w, m)}\|_{(p,m)} \lesssim  e^{\psi_m(w)}.
\end{align} Because of the reproducing property of the kernel and Parseval identity, it further  holds  that
\begin{align}
\label{kernel}
K_{(w,m)}(z)= \sum_{n=1}^\infty e_n(z) \overline{e_n(w)} \  \text{and}\ \ \|K_{_{(z,m)}}\|_{(2, m)}^2= \sum_{n=1}^\infty |e_n(w)|^2
\end{align} for any orthonormal basis $(e_n)_{n\in\NN}$ of $\mathcal{F}_{\psi_m}^{2}$.
 This and the estimate in  \eqref{asymptotic}  will be repeatedly  used  in the sequel. Another  important ingredient needed in the proofs of
the results is the following   pointwise estimate for
the reproducing kernel functions.
\begin{lemma}\label{lem1}
 There exists a small  positive number $\delta$ such that  for any $w \in \CC$
 \begin{align*}
 |K_{(m, w)}(z)| \gtrsim e^{\psi_m(z)+\psi_m(w)}
 \end{align*} for all  $z \in D(w, \delta)$, where $D(w, \delta)$ refers to the Euclidian disk of radius $r$ and center $w$.
 \begin{proof}
 The lemma will follow from \cite[Proposition~3.3]{ASDV} once we show that the weight function $\psi_m$ satisfies the growth condition
 \begin{align}
 \label{doubling}
 c \lesssim \Delta \psi_m(z) \lesssim C
 \end{align} for all $z\in \CC$ and some positive constants $c$ and $C$. Thus,  we consider
 \begin{align*}
 \psi_m(z) = \frac{|z|^2}{2}-m\log (\beta+|z|) \simeq \frac{|z|^2}{2}-\frac{m}{2}\log (\beta+|z|^2),
 \end{align*}  and a  straightforward calculation  gives  that
 \begin{align*}
 \Delta \psi_m(z) = 2-\frac{2m\beta}{(\beta+ |z|^2)^2}.
 \end{align*} Then,  the required condition \eqref{doubling} holds for any choice of
 $\beta >m $ as \\
 $2(1-\frac{m}{\beta}) \leq  \Delta \psi_m(z) \leq 2$.
 For simplicity, we will continue setting $\beta = 1 $ throughtout the rest of the paper.
 \end{proof}
\end{lemma}
\subsection{Littlewood--Paley type  formula}  Dealing with  Volterra-type integral operators  in normed spaces  gets easier when  the norms in the target spaces of the operators are   described   in terms of Littlewood--Paley type  formula. The operators have been extensively studied in the spaces where such  formulas  are  found to be accessible. The formulas will primarily help get rid of  the integrals appearing in defining the   operators. Our next key lemma does this job by characterizing  the Fock--Sobolev spaces in terms of
 derivatives.
\begin{lemma} \label{lem2}
Let $0<p<\infty$ and $f$ be a holomorphic function on $\CC$. Then
\begin{align}
\label{Paley}
\|f\|_{\mathcal{F}_{\psi_m}^p}^p \simeq |f(0)|^p + \int_{\CC}\frac{|f'(z)|^p(1+|z|)^p e^{-p\psi_m(z) }}{\big( 1+|z|+ \big| |z|^2+ |z|-m\big| \big)^p} dA(z).
\end{align}
\end{lemma}
\begin{proof}
We plan to show that the estimate in  \eqref{Paley} follows from  the   general  estimate in Theorem~19 of \cite{Olivia}. To this end, it suffices to verify that the sequence of  our weight functions $\psi_m$ satisfy all  the preconditions  required  in the theorem there, which are;\\
i) There should  exist  a positive $r_0$  for which  $\psi_m'(r) \neq 0, $ for all $r>r_0$. This, rather week requirement on the growth of $\psi_m$ works fine  as one can for example take
$$r_0= \frac{1+ \sqrt{1+ 4m}}{2}.$$
In addition, we have that $1+\psi'_m (z)\simeq \psi'_m(z)$ when $|z|\to \infty$.\\
ii) The estimates
\begin{align}
\label{requirment}
\lim_{r\to \infty} \frac{re^{-p\psi_m(r)}}{ \psi_m'(r)}= 0\nonumber\\
\limsup_{r\to \infty} \frac{1}{r} \bigg(\frac{r}{ \psi_m'(r)}\bigg)' <p \ \ \text{and} \\
\liminf_{r\to \infty} \frac{1}{r} \bigg(\frac{r}{ \psi_m'(r)}\bigg)'>-\infty, \nonumber
\end{align} hold for all positive $p.$
The first estimate in \eqref{requirment} follows easily since
\begin{align*}
\lim_{r\to \infty} \frac{re^{-p\psi_m(r)}}{ \psi_m'(r)}= \lim_{r\to \infty}  \frac{r+r^2}{r^2+r-m} e^{-p\psi_m(r)}
= \lim_{r\to \infty}  e^{-p\psi_m(r)}= 0.
\end{align*}
On the other hand, a simple computation shows that
\begin{align*}
\frac{1}{r} \bigg(\frac{r}{ \psi_m'(r)}\bigg)'= \frac{2r^2-2rm-m}{ r(r^2+r-m)^2},
\end{align*} from which  it follows that
\begin{align*}
\limsup_{r\to \infty} \frac{1}{r} \bigg(\frac{r}{ \psi_m'(r)}\bigg)' = \limsup_{r\to \infty}\frac{2r^2-2rm-m}{ r(r^2+r-m)^2}\leq 0 <p.
\end{align*}
It remains to verify the last estimate in \eqref{requirment}. But this  is rather immediate   as
\begin{align*}
\liminf_{r\to \infty} \frac{1}{r} \bigg(\frac{r}{ \psi_m'(r)}\bigg)' = \liminf_{r\to \infty}\frac{2r^2-2rm-m}{ r(r^2+r-m)^2}= 0 >-\infty.
\end{align*}
\end{proof}
We now state a key lemma on spectral properties  of  the operator $M_g$ acting between Fock--Sobolev spaces. The lemma is interest of its own.
\begin{lemma}\label{lem3}
Let $g$ be an analytic function on $\CC$ and $0<p,q< \infty$. Then if
 \begin{enumerate}
 \item  $q\geq p$, then $M_g: \mathcal{F}_{\psi_m}^p \to \mathcal{F}_{\psi_m}^q$ is bounded (compact) if and only if $g$ is a constant (zero) function.
 \item  $q< p$, then $M_g: \mathcal{F}_{\psi_m}^p \to \mathcal{F}_{\psi_m}^q$ is bounded (compact) if and only if $g$ is the zero function.
 \item  $1\leq p<  \infty$ and $M_g:\mathcal{F}_{\psi_m}^p \to \ \mathcal{F}_{\psi_m}^p$ is  a bounded map, that is $g= \alpha=$ constant, then
\begin{align*}
\sigma_p(M_g)=  \sigma(M_g)= \{\alpha \}.
\end{align*}
 \end{enumerate}
\end{lemma}
Like  that of the operator $I_g$, the lemma shows that there  exists no nontrivial multiplication operators $M_g$ acting between  the Fock--Sobolev spaces. This is rather due to the relation in \eqref{parts}, as will be also explained in the proof of \ref{thm2} in  Section \ref{partstwo}.
\begin{proof}
We observe that the multiplication operator $M_g$ is a special case of weighted composition operators $uC_\phi f(z)= u(z) f(\phi(z)); $ set $u= g$ and $\phi(z)=z.$  Several properties of  $uC_\phi$ have already  been described in
\cite{TM4} from which some will be used in our subsequent considerations.  Let us now assume that $0<p\leq q <\infty$. Then  by Theorem~3.1  of \cite{TM4}, $M_g: \mathcal{F}_{\psi_m}^p \to \ \mathcal{F}_{\psi_m}^q $ is bounded if and only
\begin{align}
\label{finite}
\sup_{w\in \CC} B_{m}(|g|^q)(w) = \sup_{w\in \CC} \int_{\CC}|k_{(w,m)}(z)|^q|g(z)|^{q}  e^{-q\psi_m(z)} dA(z)<\infty.
\end{align}
 To arrive at the desired conclusion, we may  proceed to investigate further the boundedness of the integral transform in   \eqref{finite}. To this end, assuming this condition, and applying \eqref{asymptotic} and  Lemma~\ref{lem1} we have
 \begin{align}
\label{bounded1}
\infty > \sup_{w\in \CC} \int_{\CC}\frac{ |k_{(w,m)}(z)|^q|g(z)|^{q} }{e^{q\psi_m(z)}} dA(z)
\geq \sup_{w\in \CC} \int_{D(w,\delta)} \frac{|k_{(w,m)}(z)|^q}{e^{q\psi_m(z)}} |g(z)|^{q} dA(z)\nonumber\\
\gtrsim\sup_{w\in \CC}\int_{D(w,\delta)}|g(z)|^q dA(z)
\end{align} for a small positive number $\delta$. By subharmonicity  of $|g|^q$, we further have
\begin{align}
\label{bounded2}
\infty >  \sup_{w\in \CC}\int_{D(w,\delta)}|g(z)|^q dA(z) \gtrsim   \sup_{w\in \CC} |g(w)|^{q}
\end{align} for all $w\in \CC$. From this we deduce  that
 $g$ is a bounded analytic  function on $\CC$. Then  Liouville's classical theorem forces it to be a constant.\\
Conversely, if $g$ is a constant, then the integral in \eqref{finite} is obviously finite since all the Fock--Sobolev spaces $\mathcal{F}_{\psi_m}^p$
contain the reproducing kernels (see \cite[Corollary~14]{RCKZ}).\\
A similar analysis shows that when $0<p\leq q<\infty$,  $M_g$ is compact   if and only if  $g$  is the zero  function.

(ii) When $0<q<p <\infty$, then an application of  Theorem~3.3 of \cite{TM4} ensures  that the boundedness and compactness properties of  $M_g: \mathcal{F}_{\psi_m}^p \to \mathcal{F}_{\psi_m}^q$ are equivalent  and  this happens  if and only if
\begin{align}
\label{bounded3}
\int_{\CC} \big(B_{m}(|g|^q)(w)\big)^{\frac{p}{p-q}} dA(w)\quad \quad \quad \quad \quad \quad \quad \quad \quad \quad\quad \quad \quad \quad \quad \quad \quad \nonumber\\
= \int_{\CC}\Bigg(\int_{\CC}|k_{(w,m)}(z)|^p|g(z)|^{p} e^{-q\psi_m(z)} dA(z)\Bigg)^{\frac{p}{p-q}} dA(w)<\infty.
\end{align}
Arguing as in the series of estimates leading to  \eqref{bounded1} and \eqref{bounded2}, condition \eqref{bounded3} implies
\begin{align*}
\int_{\CC} |g(w)|^pdA(w) \lesssim \int_{\CC} (B_{m}(|g|^q)(w))^{\frac{p}{p-q}} dA(w) <\infty,
\end{align*} and  from this we  conclude that $g$ is indeed  the zero function.

(iii) By part (i) of the lemma, the only bounded multiplication  operators are  the
multiplications by constant functions. It means that we are actually  dealing  with  constant multiples
of the identity operator, whose spectrum obviously consists of the multiplicative
constant.

\end{proof}
\begin{lemma}
\label{lem4}
Let $a, \lambda \in \CC, \  g(z)= az^2$  and assume that  $|\lambda| >2|a|$. If $f$ is an entire function such that $fe^{g/\lambda}$ belongs to  $\mathcal{F}_{\psi_m}^p$, then
\begin{align}
\int_{\CC} \big| e^{\frac{g(z)}{\lambda}} f(z)\big|^p e^{-p\psi_m(z)} dA(z) \lesssim  |f(0)|^p + \int_{\CC} \frac{\big|  f'(z)  e^{\frac{g(z)}{\lambda}}  \big|^p}{(1+\psi'_m(z))^p} e^{-p\psi_m(z)} dA(z).
\end{align}
\end{lemma}
 The proof  of the lemma follows from a simple variant of the proof of Proposition 1 in \cite{Olivia2}. We only need to set $\alpha=1$ and  replace
$\alpha |z|^A$ in there by $\psi_m(z)$  and reset $w(z)= p\Re(g(z)/\lambda)-p\psi_m(z)$ and run the arguments.
\section{Proof of the main results}
We now turn to the proofs of the main results of the paper. Let $0 < p,q < \infty $ and $\mu$ be a positive Borel measure on $\CC$. We call $\mu$ a  $(p, q)$ Fock--Carleson measure if  the inequality
\begin{align*}
\int_{\CC} |f(z)|^{q} e^{-\frac{q}{2}|z|^2} d\mu(z) \lesssim \|f\|_{(p,m)}^q,
\end{align*}  holds, and we call  it  a vanishing  $(p, q)$ Fock--Carleson measure if
\begin{align*}
\lim_{n\to \infty} \int_{\CC} |f_n(z)|^q e^{-\frac{q}{2}|z|^2} d\mu(z)= 0
\end{align*} for every uniformly bounded sequence $f_n$ in   $\mathcal{F}_{\psi_m}^p$   that converges to zero uniformly on compact subset of $\CC$ as $n\to \infty$. These measures have been completely identified in \cite{TM4}.\\
We observe  that by first setting $$d\mu_{(g, q)}(z)= \frac{|g'(z)|^q(1+|z|)^{qm+q }}{\big( 1+|z|+ \big| |z|^2+ |z|-m\big| \big)^q} dA(z), $$
    and applying \eqref{Paley},  we may write the norm of $V_gf$ as
  \begin{align*}
\|V_g f\|_{(q,m)}^{q}= \int_{\CC}\frac{ |g'(z)|^q|f(z))|^q (1+|z|)^{mq+q}}{\big( 1+|z|+ \big| |z|^2+ |z|-m\big| \big)^q} e^{-\frac{q}{2}|z|^2}dA(z)\quad \quad \quad \nonumber\\
=  \int_{\CC} |f(z)|^q e^{-\frac{q}{2}|z|^2} d \mu_{(g,q)}(z).
 \end{align*}
  In view of this,  it follows that $V_g:  \mathcal{F}_{\psi_m}^p \to \mathcal{F}_{\psi_m}^q $ is bounded (compact) if and only if
 $\mu_{(g,q)}$  is a (p,q) (vanishing) Fock-Carleson measure.  Consequently, if
   $0\leq p\leq q<\infty,$ then Theorem~2.1 of \cite{TM4} ensures that   $\mu_{(g,q)}$  is a (p,q)  Fock-Carleson measure  if and only if  $\tilde{\mu}_{(t, mq)}$ is bounded for some or any positive $t$ where
 \begin{align*}
\tilde{\mu}_{(t, mq)}(w)=\int_{\CC} \frac{e^{-\frac{t}{2}|z-w|^2}}{ (1+|z|)^{mq}} d\mu_{(g, q)}(z).
 \end{align*}
 Having singled out this equivalent reformulation, our next task will be to   investigate the new formulation further,  namely boundedness of the transform  $\tilde{\mu}_{(t, mq)}$. Let us  first assume its boundeness,   and show that $g$ is a complex polynomial of  degree  not exceeding two. To this end,
\begin{align*}
\infty > \sup_{w\in \CC}\int_{\CC} \frac{e^{-\frac{t}{2}|z-w|^2}}{ (1+|z|)^{mq}} d\mu_{(g, q)}(z)=\sup_{w\in \CC} \int_{\CC}  \frac{e^{-\frac{t}{2}|z-w|^2} |g'(z)|^q(1+|z|)^{q }}{\big( 1+|z|+ \big| |z|^2+ |z|-m\big| \big)^q} dA(z)\nonumber\\
\gtrsim \sup_{w\in \CC} \int_{D(w, 1)} \frac{|g'(z)|^q(1+|z|)^{q }}{\big( 1+|z|+ \big| |z|^2+ |z|-m\big| \big)^q} dA(z)=S.
\end{align*}
Observe that whenever  $z$ belongs to the disk $ D(w,1)$, then
\begin{align}
\label{comparison}
\begin{cases}
1+|z| \simeq |+|w|  &\   \\
\ 1+|z|+\big| |z|+|z|^2-m\big| \simeq 1+|w|+\big| |w|+|w|^2-m\big|.
\end{cases}
\end{align}
This together with the subharmonicity of $|g'|^q$  implies that
\begin{align*}
S\gtrsim \frac{|g'(w)|^q(1+|w|)^{q }}{\big( 1+|w|+ \big| |w|^2+ |w|-m\big| \big)^q} ,
\end{align*}
  for all $w\in \CC$ and  our  assertion follows. \\
  On the other hand, if $g(z)= az^2+bz+c, \ a, b, c\in \CC$, then
  \begin{align*}
  \sup_{w\in \CC }\tilde{\mu}_{(t, mq)}(w)= \sup_{w\in \CC }\int_{\CC} e^{-\frac{t}{2}|z-w|^2}\frac{|2az+b|^q(1+|z|)^{q }}{\big( 1+|z|+ \big| |z|^2+ |z|-m\big| \big)^q}dA(z)\nonumber\\
   \lesssim \sup_{w\in \CC } \int_{\CC} e^{-\frac{t}{2}|z-w|^2}dA(z) < \infty,
\end{align*} and completes the proof of part (a)  of (i)  in the theorem.\\
Similarly,  for part (b),  for $0<p\leq q<\infty,$  by Theorem~2.2 of \cite{TM4}, $\mu_{(g,q)}$  is a (p,q) vanishing Fock-Carleson measure if and only if $\tilde{\mu}_{(t, mq)}(z) \to 0$ as $|z| \to \infty$. We may first assume this vanishing property and show that $g$ is a complex polynomial of degree not exceeding one. For this, following the same arguments as above, we easily see from our assumption that
\begin{align*}
\lim_{|w| \to \infty }\frac{|g'(w)|(1+|w|)}{1+|w|+ \big| |w|^2+ |w|-m\big| }= 0,
\end{align*}  and this obviously  holds only if $g'$  is a constant as asserted.\\
Conversely, if $g(z)= az+b, \ a, b, \in \CC$, then
\begin{align*}
\lim_{|w| \to \infty}\tilde{\mu}_{(q, mq)}(w)= \lim_{|w| \to \infty} \int_{\CC} e^{-\frac{q}{2}|z-w|^2}\frac{|a|^q(1+|z|)^{q }}{\big( 1+|z|+ \big| |z|^2+ |z|-m\big| \big)^q}dA(z)\nonumber\\
\lesssim  \lim_{|w| \to \infty} \int_{\CC} \frac{ e^{-\frac{q}{2}|z-w|^2}}{(1+|z|)^{q }}dA(z) \simeq \lim_{|w| \to \infty} (1+|w|)^{-q}= 0.
\end{align*}
ii) If $0<q<p<\infty$, then  by Theorem~2.3 of \cite{TM4} again, $V_g:  \mathcal{F}_{\psi_m}^p \to \mathcal{F}_{\psi_m}^q$ is bounded (compact) if and only if $\tilde{\mu}_{(t, mq)}$ belongs to $L^{\frac{p}{p-q}}(\CC, dA)$. We plan to show that  this holds if and only if $g$ is  of at most degree one and $q> \frac{2p}{p+2}$. To this end, applying \eqref{comparison} and subharmonicity of $|g'(w)|^{\frac{pq}{p-q}}$, we infer
\begin{align*}
\int_{\CC}|g'(w)|^{\frac{pq}{p-q}}\Bigg(\frac{1+|w| }{ 1+|w|+ \big| |w|^2+ |w|-m\big| }\Bigg)^{\frac{pq}{p-q}} dA(w) \quad \quad \quad \quad \quad \quad \quad \quad \quad \quad  \quad \quad \nonumber\\
\lesssim \int_{\CC}\Bigg( \int_{D(w, 1)}|g'(z)|^q e^{-\frac{q}{2}|z-w|^2} \Bigg(\frac{1+|z| }{ 1+|z|+ \big| |z|^2+ |z|-m\big| }\Bigg)^{q} dA(z)\Bigg)^{\frac{p}{p-q}} dA(w)\nonumber\\
\leq \int_{\CC}\bigg( \int_{\CC}|g'(z)|^q e^{-\frac{q}{2}|z-w|^2} \bigg(\frac{1+|z| }{ 1+|z|+ \big| |z|^2+ |z|-m\big| }\bigg)^{q} dA(z)\bigg)^{\frac{p}{p-q}} dA(w)\nonumber\\
\simeq  \int_{\CC} \tilde{\mu}_{(q, mq)}^{\frac{p}{p-q}} (w)dA(w) <\infty,
\end{align*}
from which we conclude that $g'$ must be a constant. In addition, if $g'$ is a nonzero constant, the above holds only if $\frac{pq}{p-q}>2.$ \\
Conversely, assuming that $g'\simeq \alpha=$ constant  we have
\begin{align*}
 \int_{\CC} \tilde{\mu}_{(q, mq)}^{\frac{p}{p-q}} (w)dA(w) \simeq \int_{\CC}\Bigg( \int_{\CC} \Bigg(\frac{|\alpha| e^{-\frac{1}{2}|z-w|^2} (1+|z|) }{ 1+|z|+ \big| |z|^2+ |z|-m\big| }\Bigg)^{q} dA(z)\Bigg)^{\frac{p}{p-q}} dA(w)\nonumber\\
 \lesssim  \int_{\CC}  \Bigg( \int_{\CC} \frac{|\alpha|^qe^{-\frac{q}{2}|z-w|^2}}{(1+|z|)^q} dA(z)\Bigg)^{\frac{p}{p-q}} dA(w)
 \simeq \int_{\CC} |\alpha|^q(1+|w|)^{-\frac{pq}{p-q}} dA(w) <\infty,
\end{align*} where the last integral converges since either $\frac{pq}{p-q}>2$ or $\alpha=0$ by our assumption.

(iii) Let us now turn to the Schatten class membership of $V_g$. We recall that a compact  operator  $V_g$ belongs
to the Schatten $\mathcal{S}_p(\mathcal{F}_{\psi_m}^{2})$ class if and only if the sequence of the eigenvalues of  the positive operator $(V_g^*V_g)^{1/2}$  is  $\ell^p$ summable.  In particular when $p\geq 2,$ this happens  if and only if
\begin{align}
\label{sum}
\sum_{n=1}^\infty \|Te_n\|_{(2,m)}^p <\infty
\end{align} for any orthonormal basis $(e_n)$ of  $\mathcal{F}_{\psi_m}^{2}$ (see \cite[Theorem 1.33]{KZH1}).  Let us assume that  $V_g$ is compact, that is $g'$ is a constant.  Then for
 $ p>2$, applying \eqref{sum}, \eqref{Paley},  and H\"{o}lder's inequality,  and subsequently \eqref{kernel} and \eqref{asymptotic}, we  compute
\begin{align*}
\sum_{n=1}^\infty \|V_g e_n\|_{(2, m)}^p \simeq \sum_{n=1}^\infty \Bigg( \int_{\CC} \bigg(\frac{|e_n(z)| e^{-\psi_m(z)}(1+|z|) }{ 1+|z|+ \big| |z|^2+ |z|-m\big| }\bigg)^{2} dA(z) \Bigg)^{\frac{p}{2}} \quad \quad \quad \quad  \quad \quad \quad \quad \nonumber\\
 \leq \sum_{n=1}^\infty  \Bigg( \int_{\CC}\frac{|e_n(z)|^2 (1+|z|)^2}{ e^{2\psi_m(z)} }   dA(z)\Bigg)^{\frac{p-2}{2}}
 \int_{\CC} \frac{ |e_n(z)|^2 e^{-2\psi_m(z)} \big(1+|z|) ^2}{(1+|z|+ \big| |z|+|z|^2 -m\big|\big)^p}dA(z)\ \ \quad  \nonumber\\
\simeq \sum_{n=1}^\infty \int_{\CC} \frac{ |e_n(z)|^2 e^{-2\psi_m(z)}   \big(1+|z|)^2}{(1+|z|+ \big| |z|+|z|^2 -m\big|\big)^p}dA(z)\quad \quad\nonumber\\
\simeq \int_{\CC}\Big(1+|z|+ \big| |z|+|z|^2 -m\big|\Big)^{-p}dA(z)<\infty.\quad \quad \quad \quad
\end{align*}
On the other hand, if p= 2 and $g'= \alpha=$ constant, then
\begin{align}
\sum_{n=1}^\infty \|V_g e_n\|_{(2, m)}^2 \simeq \sum_{n=1}^\infty \int_{\CC}\Bigg(\frac{ |g'(z)||e_n(z)| e^{-\psi_m(z)}(1+|z|) }{ 1+|z|+ \big| |z|^2+ |z|-m\big| }\Bigg)^{2}dA(z) \quad \quad \nonumber\\
 \simeq \int_{\CC}|\alpha|^2 \big(1+|z|+ \big| |z|^2+ |z|-m\big|\big)^{-2}dA(z).\nonumber
\end{align}The last integral above is finite if and only if $\alpha= 0$, and hence $g$ is a constant. The same conclusion holds for the case when $0<p<2$ by the monotonicity property of Schatten class membership, in the sense that
$\mathcal{S}_p(\mathcal{F}_{\psi_m}^2) \subseteq \mathcal{S}_2(\mathcal{F}_{\psi_m}^2),$ for all $p\leq2$.\\
\textbf{Remark 1:} Because of Lemma~\ref{lem2}, it is tempting to prove Theorem~\ref{thm1} by first  setting
\begin{align*}
\int_{\CC} |V_gf(z)|^q e^{-q\psi_m(z)} dA(z) \simeq \int_{\CC} |f(z)|^q \frac{|g'(z)|^q(1+|z|)^q  e^{-q\psi_m(z)}}{\big(1+|z|+\big| |z|+|z|^2-m\big|\big)^q} dA(z)\nonumber\\
\simeq \int_{\CC} |M_{h}f(z)|^qe^{-q\psi_m(z)} dA(z),
\end{align*} where $$h(z)= \frac{|g'(z)|(1+||z)}{1+|z|+\big| |z|+|z|^2-m\big|},$$ and then apply  Lemma~\ref{lem3} with $h$ as the multiplier function. Unfortunately, this approach is not valid as the  lemma on the multiplication operator can not be  directly applied; since the  analyticity property of $g$ is heavily  used in its proof, while $h$ fails to be analytic in here.
\subsection{Proof of Theorem~\ref{thm2}}\label{partstwo}
We note that relation \eqref{parts} ensures that  if any two of the operators are bounded (compact), so is  the third one.  In view of this, $I_g: \mathcal{F}_{\psi_m}^p \to \mathcal{F}_{\psi_m}^q$ is bounded (compact)  if both $M_g$ and $V_g$ are bounded (compact).  By Theorem~\ref{thm1} and Lemma~\ref{lem3}, this happens if and only if  $g$ is a constant function.  This obviously gives the sufficiency part of the conditions in the theorem.  We  proceed to show that it is also necessary. First from Lemma~\ref{lem1}, Cauchy--Schawarz inequality and \eqref{asymptotic}, observe that for each $z\in D(w, \delta)$ and a small positive $\delta$;
\begin{align*}
| K_{(w, m)}(z)| \simeq e^{\psi_m(z)+\psi_m(w)}.
\end{align*}
 Now assuming that $I_g: \mathcal{F}_{\psi_m}^p \to \mathcal{F}_{\psi_m}^q$ is bounded and  $0<p\leq q<\infty. $ Then,  applying \eqref{Paley}, \eqref{forall}  we have
\begin{align*}
e^{q\psi_m(w)} \gtrsim \int_{\CC} |I_g K_{(w, m)}(z)|^q  e^{-q\psi_m(z)} dA(z) \quad \quad \quad \quad \quad \quad  \quad \quad \quad   \quad \quad \quad \quad \nonumber\\
\simeq \int_{\CC} \frac{|K'_{(w, m)}(z)|^q  |g(z)|^q (1+|z|)^{q}}{\big(1+|z|+\big||z|+|z^2|-m\big|\big)^q} e^{-q\psi_m(z)} dA(z)\nonumber\\
\geq \int_{D(w,\delta)}\frac{|K'_{(w, m)}(z)|^q  |g(z)|^q (1+|z|)^{q}}{\big(1+|z|+\big||z|+|z^2|-m\big|\big)^q} e^{-q\psi_m(z)} dA(z)\nonumber\\
\gtrsim \int_{D(w,\delta)}\frac{e^{q\psi_m(w)}|\psi'_m(z)|^q |g(z)|^q (1+|z|)^{q}}{\big(1+|z|+\big||z|+|z^2|-m\big|\big)^q}  dA(z)= S_1.
\end{align*}
On the other hand $ \psi_m'(z)\simeq \psi_m(w)$ for each $z\in D(w, \delta)$. Applying  this, \eqref{comparison}, and the subharmonicity of $|g|^q$, we estimate $S_1$ from below as
\begin{align*}
S_1\gtrsim  \frac{e^{q\psi_m(w)}|\psi'_m(w)|^q  |g(w)|^q (1+|w|)^{q}}{\big(1+|w|+\big||w|+|w^2|-m\big|\big)^q}  \simeq  \frac{e^{q\psi(w)}\Big( |w|+|w|^2 -m\Big)^q |g(w)|^q}{\big(1+|w|+\big||w|+|w^2|-m\big|\big)^q},
\end{align*} from which and taking further simplifications, we infer
\begin{align*}
|g(w)| \lesssim \frac{1+|w|}{|w|+|w|^2-m}+ 1,
\end{align*} and hence g is a bounded analytic function.  By  Liouville's  classical theorem, $g$ turns out to be a constant function.
\subsection{Proof of Theorem~\ref{thm3}}
Recall that $\lambda \in \CC$  belongs to the spectrum $\sigma(T)$ of a bounded  operator $T$ on a  Banach space if  $\lambda I-T$ fails to be invertible, where $I$ is the identity operator on the space. The point spectrum $\sigma_p(T)$  of $T$ consists of  its eigenvalues.  We now turn to the spectrum of $V_g$ in particular, and  assume that  $V_g$ is bounded on $\mathcal{F}_{\psi_m}^p$ and hence $g(z)= az^2+ bz+c, \ \  a, b, c \in \CC.$ By linearity of integrals we  may first  make a  splitting  $\lambda I-V_g = (\lambda I-V_{g_{1}})-V_{g_{2}}$ where $g_{1}(z)= az^2$ and $g_{2}(z)= bz+c$. A simple analysis shows that $\lambda I- V_g$ and $ \lambda I- V_{g_1}$ are injective maps. On the other hand, by part (i) of the result, $V_{g_{2}}$ is compact and hence $\sigma (V_{g_{2}})= \{0\}.$  Thus, we shall  investigate the case with $V_{g_{1}}$.  We may first observe that  if  $\lambda\neq 0,$  then the equation  $\lambda f-V_g f= h$  has the unique analytic solution
\begin{align}
\label{first}
f(z)=(\lambda I- V_{g_1})^{-1}h(z)= \frac{1}{\lambda} h(0) e^{\frac{g_1(z)}{\lambda}} + \frac{1}{\lambda}  e^{\frac{g_1(z)}{\lambda}} \int_{0}^z  e^{-\frac{g_1(w)}{\lambda}} h'(w) dA(w),
\end{align} where I is the identity operator.   This can easily be  seen by solving an initial valued  first order linear ordinary differential equation
 $$ \lambda y'- g_1'y= h', \ \ \lambda f(0)= h(0). $$
 Recall that $(\lambda I- V_{g_1})^{-1}h(z)=   R_{(g_1,\lambda)} h(z)$
  is the Resolvent operator of $V_{g{1}}$ at $\lambda$. It follows that $\lambda \in \CC$ belongs to the resolvent of $V_{g_1}$ whenever  $R_{(g_1,\lambda)}$ is a bounded operator. Since we assumed that $V_{g_1}$ is bounded and as  $\mathcal{F}_{\psi_m}^p $ contain the  constants, setting $h= 1$ in \eqref{first} shows that $R_{(g_1,\lambda)} 1= e^{g_1(z)/\lambda} \in \mathcal{F}_{\psi_m}^p $  for each $\lambda$ in the resolvent set of $V_{g_1}$. From this, we obviously deduce
      \begin{align*}
   \sigma(V_{g_1}) \supseteq \{0\}\cup \overline{ \{\lambda \in \CC\setminus\{0\}:e^{g_1(z)/\lambda} \notin \mathcal{F}_{\psi_m}^p\} }.
   \end{align*}
On the other hand, if $|\lambda|> 2|a|$, then we set polar coordinates for  $z= re^{i\theta}, \ \ a= |a|e^{i\theta_1}, \ \ \lambda= |\lambda|e^{i\theta_2}$,  and estimate
\begin{align*}
\int_{\CC} |R_{(g_1,\lambda)} 1(z)|^p e^{-p\psi_m(z)}dA(z)= \int_{\CC} e^{p \Re\big(\frac{az^2}{\lambda}\big)-p\psi_m(z)} dA(z)\quad \quad \quad \quad \\
=\int_{0}^\infty \int_{0}^{2\pi} e^{ p\big( \frac{|a|}{|\lambda|} \cos(\theta+ \theta_1-\theta_1)-\frac{1}{2}\big)r^2 + mp\log(1+r) } r d\theta dr\quad \quad \quad \\
\lesssim \int_{0}^\infty  e^{ p\big( \frac{|a|}{|\lambda|} -\frac{1}{2}\big)r^2 + (m+1)p\log(1+r) } dr
\lesssim \int_{0}^\infty  e^{ p\big( \frac{|a|}{|\lambda|} -\frac{1}{2}\big)r^2 + p(m+1)r } dr\\
\leq \sqrt{ \frac{2\pi |\lambda|}{p(2|a|-|\lambda|)} } e^{\frac{2|\lambda|(pm+p)^2 }{p(2|a|-|\lambda|)}} <\infty.\quad \quad \quad
\end{align*}
 This means that  the spectrum of $V_{g_1}$ contains the closed disc $\overline{D(0, 2|a|)}.$  We remain to show that $R_{(g_1, \lambda)}$ is bounded  for all $\lambda \in \CC$ such that $|\lambda|\leq 2|a|$.  To this end, applying Lemma~\ref{lem2} and Lemma~\ref{lem4}, we have
 \begin{align*}
 \int_{\CC} |R_{(g_1,\lambda)}f(z)|^p e^{-p\psi_m(z)} dA(z) \leq  2^p |f(0)|^p\int_{\CC} |e^\frac{g_1(z)}{\lambda}|^p e^{-p\psi_m(z)}dA(z) \quad \quad \quad \nonumber\\
 + 2^p  \int_{\CC} \bigg|e^{-\frac{g_1(z)}{\lambda}}\int_{0}^z e^{-\frac{g_1(w)}{\lambda}}f'(w)dA(w)\bigg|^p e^{-p\psi_m(z)} dA(z)\nonumber\\
 \lesssim \|f\|_{(p,m)}^p + \int_{\CC}\frac{|f'(z)|^pe^{-p\psi_(z)}}{(1+\psi'_m(z))} dA(z) \lesssim  \|f\|_{(p,m)}^p,
 \end{align*} and completes the proof of par(t i).

  The proof of part (ii) is rather  straightforward. If $\lambda$ belongs to the  point spectrum of $I_g$, then there exists a nonzero function $f \in\mathcal{F}_{\psi_m}^p$ for which
 \begin{align}
 \label{easy}
 \lambda f(z)= \int_{0}^z c f'(w)dA(w)
 \end{align} It follows from this that $\lambda f'(z)= cf'(z)$ which holds either $\lambda= c$ or  $f'= 0$. The later leads to  a contradiction
 because of the relation in  \eqref{easy}. Thus, we  must have $\lambda= c$. This implies
 $$\{c\}\subseteq \sigma(I_g).$$
 To show the converse inclusion, it suffices to show that the resolvent operator $R_{(\lambda,g)}$ of  $I_g$ at point $\lambda$ is  bounded on $\mathcal{F}_{\psi_m}^p$  for each $\lambda \neq c.$ To this end, from the relation $\lambda f-I_g f= h,$ it follows that
 $$ \lambda f'-cf'= h'.$$
 Solving this linear ordinary differential equation gives the explicit expression for the resolvent operator
\begin{align*}
f(z)= R_\lambda h(z)= \frac{h(z)}{\lambda-c}, \end{align*} which obviously is bounded on $\mathcal{F}_{\psi_m}^p$ and completes the required proof.
 \subsection{The differential operator D} The  differential operator $Df= f'$ has become a prototype example of unbounded operators  in many Banach spaces. Its unboundedness in the classical Flock spaces with Gaussian weight and in weighted Fock spaces where the weight decays faster than the Gaussian weight  was recently verified in \cite{TM3}. Another natural question would be then what happens when the weight decays slower than the Gaussian weight in which the Fock--Sobolev spaces constitute  typical examples. In what follows we will verify that the action of  the  operator remains unbounded.
If $D:\mathcal{F}_{\psi_m}^p \to \mathcal{F}_{\psi_m}^q $  were indeed bounded, then applying $D$ to the sequence of the reproducing kernels, using estimates  \eqref{asymptotic} and \eqref{forall},  and subharmonicity of $| K'_{(w, m)}|^q$, we would find
\begin{align*}
e^{q\psi_m(w)} \gtrsim  \| K'_{(w, m)}\|_{(p,m)}^q \|D\|^q \geq \int_{\CC} | K'_{(w, m)}(z)|^q  e^{-q\psi_m(z)} dA(z)\nonumber\\
 \geq \int_{D(w,1)}| K'_{(w, m)}(z)|^q  e^{-q\psi_m(z)} dA(z)  \nonumber\\
\gtrsim | K'_{(w, m)}(w)|^q  e^{-q\psi_m(w)} \simeq  \psi'(w) e^{q\psi_m(w)}
\end{align*} and from this we conclude
$ ||w|+|w|^2 -m|\lesssim 1+|w|$,
 resulting  the desired contradiction when $|w| \to \infty.$


\begin{thebibliography}{BRSHZE}

\bibitem{ALC} A. Aleman and J.  Cima, An integral operator on $H^p$ and Hardy's inequality,  J. Anal.
Math., \textbf{85}  (2001), 157--176.

\bibitem{Alsi1} A. Aleman  and A. Siskakis, An integral operator on $H^p$,  Complex Variables,  \textbf{28} (1995), 149--158.

\bibitem{Alsi2} A. Aleman  and A. Siskakis, Integration operators on Bergman spaces,  Indiana University Math J.,  \textbf{46} (1997), 337--356.


    \bibitem{CCK} H.R. Cho, B.R. Choe, and  H. Koo,Linear combinations of composition operators on the Fock--Sobolev spaces, Potential Anal., \textbf{41} (2014),  1223--1246.
\bibitem{RCKZ} R. Cho  and  K. Zhu. Fock--Sobolev spaces and their Carleson measures, J. Funct. Anal.,  \textbf{15}  (2012), 2483--2506.

\bibitem{Olivia1} O. Constantin, Volterra type integration operators on Fock spaces,  Proc. Amer. Math. Soc., (12) \textbf{140}  (2012),
4247--4257.
\bibitem{Olivia2} O. Constantin and Ann-Maria  Persson, The spectrum of Volterra type integration operators on generalized  Fock spaces,
 Bull. London Math. Soc., 47(2015), 6, 958--963.
\bibitem{Olivia} O. Constantin and Jos\'{e} \'{A}ngel Pel\'{a}ez, Integral Operators, Embedding Theorems and a Littlewood--Paley Formula on Weighted Fock Spaces,  J. Geom. Anal.,  (2015) 1--46.


\bibitem{GGP} P. Galanopoulos, D. Girela and J. A. Pelaez, Multipliers and integration operators on Dirichlet
spaces, Trans. Amer. Math. Soc., \textbf{363} (2011), 1855--1886.
\bibitem{GP} D. Girela and J.A. Pelaez, Carleson measures, multipliers and integration operators for spaces
of Dirichlet type, J. Funct. Anal., \textbf{241} (2006), 334--358.

\bibitem{TM3} T. Mengestie and S.  Ueki,  Integral, Differential  and multiplication operators on weighted Fock spaces, Preprint, 2015.
    \bibitem{TM4} T. Mengestie, Carleson type measures for Fock--Sobolev spaces,
Complex Anal. Oper. Theory,  \textbf{8} (2014), no 6, 1225--1256.

\bibitem{TM1} T. Mengestie, Generalized Volterra companion operators on Fock spaces, Potential Anal., (2015),  1--20.

\bibitem{TM} T. Mengestie, Product of Volterra type integral and composition operators on weighted Fock spaces,
 J. Geom. Anal., \textbf{24} (2014), 740--755.
 \bibitem{TM2} T. Mengestie, Schatten class generalized Volterra companion operators on Fock spaces,  Banach J. Math. Anal., \textbf{10}(2016), no.2,  267--280.
 \bibitem{TMM} T. Mengestie,   Volterra type integral operators  and composition operators on Model spaces,
      Journal of Function Spaces, vol. \textbf{2015}, doi:10.1155/2015/467802.
\bibitem{TM0} T. Mengestie, Volterra type and weighted composition operators on  weighted Fock spaces,  Integr. Equ. Oper.
Theory, \textbf{76} (2013), no. 1, 81--94.

\bibitem{JPP} J. Pau and J. A. Pel\'{a}ez,  Embedding theorems and  integration operators on Bergman spaces  with rapidly decreasing weightes. J. Funct. Anal.,  \textbf{259 (10)}(2010), 2727--2756.
    \bibitem{JPP1} J. Pau and J. A. Pel\'{a}ez, Volterra type operators on Bergman spaces with exponential weights,  Contemporary Mathematics,  \textbf{561}(2012), 239--252.
    \bibitem{Jord} J. Pau, Integration operators between Hardy spaces of the unit ball of $\CC^n$, J. Funct. Anal., \textbf{270} (2016), 134--176
    \bibitem{JR} J. A. Pelaez and J. R\"{a}tty\"{a}, Weighted Bergman spaces induced by rapidly increasing weights, Mem. Amer. Math. Soc.,  \textbf{227} (2014), no. 1066.
\bibitem {ASDV} A. Schuster and D, Varnolin, Teoplitz operators and Carleson measures on generalized Bergmann--Fock spaces, Integr. Equ. Oper.
Theory, \textbf{72} (2012, 363--392.)
    \bibitem{KZH1} K. Zhu, Operator theory on function spaces, Second Edition, Math. Surveys and
Monographs, Vol. 138, American Mathematical Society: Providence, Rhode Island, 2007.
\bibitem{XH} X. Zhu,  Volterra composition operators on logarithmic bloch spaces, Banach J. Math. Anal.,  \textbf{3}(2009), 122--130.

\end{thebibliography}
\end{document}